\documentclass{jams-l} 

\makeatletter
\def\@printyear{TT}
\makeatother

\issueinfo{25}{4}{October}{2012}
\pagespan{929}{958}
\dateposted{May 1, 2012}
\PII{S 0894-0347(2012)00737-1}
\copyrightinfo{2012}{American Mathematical Society}

\revertcopyright
\usepackage{url}
\usepackage{breakurl}

\makeatletter
\newsavebox\tboxa
\newsavebox\tboxb
\newlength\tdima
\newcommand*{\oversymb}{\mathpalette\@oversymb}
\newcommand*{\@oversymb}[2]{%
    \sbox{\tboxa}{$\m@th#1\mathrm{#2}$}%
    \setbox\tboxb\null%
    \ht\tboxb\ht\tboxa%
    \dp\tboxb\dp\tboxa%
    \wd\tboxb\wd\tboxa%
    \sbox{\tboxa}{$\m@th#1{#2}$}%
    \setlength\tdima{\the\wd\tboxa}%
    \addtolength\tdima{-\the\wd\tboxb}%
    \sbox{\tboxb}{$\m@th#1\hskip\tdima\overline{\xusebox{\tboxb}}$}%
    \rlap{\usebox\tboxb}{\usebox\tboxa}}
\newcommand*{\xusebox}[1]{\mathord{{\usebox{#1}}}}

\newcommand{\C}{\mathbb{C}}
\newcommand{\R}{\mathbb{R}}
\newcommand{\Ps}{\mathbb{P}}

\newcommand{\N}{\mathbb{N}}
\newcommand{\Hermite}{\mathcal{H}}

\newcommand{\trace}{\mathop{\textup{tr}}}
\newcommand{\co}{\colon}

\newcommand{\mult}{\mathop{\textup{mult}}\nolimits}
\newcommand{\Sym}{\mathop{\textup{Sym}}}
\newcommand{\Hom}{\mathop{\textup{Hom}}}
\newcommand{\Stab}{\mathop{\textup{Stab}}}

\theoremstyle{plain}
\newtheorem{theorem}{Theorem}
\newtheorem{corollary}[theorem]{Corollary}
\newtheorem{lemma}[theorem]{Lemma}

\newtheorem{conjecture}[theorem]{Conjecture}

\theoremstyle{remark}

\theoremstyle{definition}

\numberwithin{equation}{section}

\begin{document}

\title{Three-point bounds for energy minimization}

\author{Henry Cohn}
\address{Microsoft Research New England,
One Memorial Drive,
Cambridge, Massachuetts 02142}
\email{cohn@microsoft.com}

\author{Jeechul Woo}
\address{Department of Mathematics,
Harvard University,
Cambridge, Massachusetts 02138}
\email{woo@math.harvard.edu}
\thanks{The second author was supported in part by an internship at
Microsoft Research New England and by a Samsung Scholarship.}

\date{March 2, 2011 and, in revised form, February 24, 2012}
\subjclass[2010]{Primary 05B40, 52A40, 52C17; Secondary 90C22, 82B05}


\maketitle

\vskip -4pt \vglue-12pt
\section{Introduction}

Consider the seven lines connecting opposite vertices of a cube and of
its dual octahedron.  Although the symmetry group does not act
transitively on the lines, they are exceedingly well distributed within
$\R\Ps^2$.  In this paper, we prove that they form a universally
optimal configuration; in other words, they minimize a wide variety of
natural notions of energy. Universal optima are rare, and we show that
this configuration is the largest universal optimum in $\R\Ps^2$.

To prove universal optimality, we use semidefinite programming bounds,
which are a powerful technique for proving bounds in coding theory.  We
give a new derivation of these bounds and extend them from coding to
energy minimization.  Proving universal optimality involves challenges
that have not arisen in previous applications of semidefinite
programming bounds, and we provide a general methodology for meeting
these challenges.  Furthermore, we conjecture that in certain other
cases our bounds remain sharp throughout a phase transition between two
different ground states, which would be a remarkable phenomenon.  We
have not yet been able to prove these conjectures, but the techniques
we introduce here represent the first steps in a program to do so.

\subsection{Background}

What does it mean to distribute $N$ points as uniformly as possible in
a compact metric space $X$ with metric $d$?  There are many possible
answers, such as forming a good error-correcting code, i.e., maximizing
the distance between the closest points.  One particularly important
family of answers generalizing coding theory is given by potential
energy minimization. Given a decreasing, continuous function $f \co
\big(0,\max_{x,y \in X} d(x,y)^2\big] \to \R$, called the
\emph{potential function}, define the \emph{energy} of a finite
configuration $\mathcal{C} \subseteq X$ by
$$
E_f(\mathcal{C}) = \frac{1}{2}
\sum_{\shortstack[c]{$\scriptstyle x,y \in \mathcal{C}$\\
$\scriptstyle x \ne y$}}
f\big(d(x,y)^2\big).
$$
Note that the use of squared distance instead of distance is not
standard in physics, but it is mathematically convenient (and of course
there is no loss of generality). We wish to choose $\mathcal{C}$ so as
to minimize $E_f(\mathcal{C})$ subject to $|\mathcal{C}|=N$. Because
$f$ is decreasing, this amounts to moving the points as far apart as
possible, but all the distances matter, not just the minimal distance.

This energy minimization problem arises naturally in physics, as the
problem of determining the ground state of a classical particle system
with isotropic pair interactions.  Even in the case of particles
confined to two-dimensional surfaces, these models capture important
features of many real-world materials, such as colloidal particles
adsorbing to the surface of a droplet (see \cite{BG} for details and
other examples).

Small systems frequently display beautiful symmetry, such as twelve
particles on a sphere forming the vertices of an icosahedron, but in
larger systems the symmetry is often broken by the appearance of
defects.  One might expect that these defects would occur only in local
minima for energy, and that the global minimum would have a perfect
crystalline structure, with a large symmetry group. Sometimes that is
the case, but in many systems the defects actually contribute to
lowering energy. (See for example the discussion of spherical crystals
in Section~3 of \cite{BG}.) One of the fundamental problems in this
area is understanding when highly symmetrical configurations are
optimal and why; see \cite{C} for a survey. Even when the answer is
easy to guess, it is usually not easy to prove, and guessing the answer
can itself be tricky.

Energy minimization can also be viewed as a generalization of coding
theory.  Specifically, it includes coding theory as a degenerate
special case: if we take $f(r) = 1/r^s$, then as $s \to \infty$, the
problem of minimizing $E_f(\mathcal{C})$ turns into the problem of
maximizing the minimal distance in $\mathcal{C}$.  (Ties are broken by
how many times this distance occurs, then what the next smallest
distance is, then how many times it occurs, etc.) The same problem of
understanding symmetry occurs here as well: when should one expect an
optimal code to be highly symmetrical? Large or high-dimensional codes
are frequently much less highly structured than small codes.

Linear programming bounds are a powerful technique for proving lower
bounds for energy.  In the simplest cases, they deal with the
\emph{distance distribution} of the configuration (in physics terms,
the radial pair correlation function), which counts how many times each
distance occurs between pairs of points. In other words, the distance
distribution of $\mathcal{C}$ is the function $\delta \co [0,\infty)
\to \R$ defined by
$$
\delta(r) =
\left|\{x,y \in \mathcal{C} : d(x,y) = r\}\right|.
$$
Clearly $\delta(0)=|\mathcal{C}|$, $\delta(r) \ge 0$ for all
$r$, and $\sum_r \delta(r) = |\mathcal{C}|^2$ (note that only
finitely many terms are nonzero), but we will see that these
are far from the only constraints on $\delta$.  The distance
distribution plays a key role because energy is a linear
function of $\delta$:
$$
E_f(\mathcal{C}) = \frac{1}{2}\sum_{r > 0} f\big(r^2\big)\, \delta(r).
$$
In contrast with its energy, the underlying configuration $\mathcal{C}$
is not always determined by $\delta$ (see \cite{Pa}).

Delsarte \cite{D} realized that in addition to the obvious constraints
on $\delta$ mentioned above, there are often many other linear
constraints.  Using these constraints, one can derive bounds on energy
via linear programming, because one is optimizing a linear function of
$\delta$ subject to linear constraints. This was first carried out for
energy minimization by Yudin \cite{Y}. We refer to them as
\emph{two-point bounds} because they depend only on the distances
between pairs of points.  Linear programming bounds are generally not
sharp, or even close to sharp, at least as best we can judge based on
the available evidence.  However, in certain cases they are
unexpectedly powerful. For example, Cohn and Kumar \cite{CK} used
linear programming bounds to prove that a number of exceptional
structures in spheres or projective spaces are universally optimal.  In
other words, these configurations minimize energy for all completely
monotonic functions of squared Euclidean distance on spheres or squared
chordal distance in projective space. (Recall that a completely
monotonic function is a smooth, nonnegative function whose derivatives
alternate in sign: it is decreasing, convex, etc.  For example, inverse
power laws are completely monotonic. See \cite[p.~101]{CK} for
motivation for the use of squared chordal distance, and
\cite[pp.~107--108]{CK} for counterexamples to a natural
generalization.) Examples include the vertices of any regular polytope
with simplicial facets, the $E_8$ root system, or the minimal vectors
in the Leech lattice.

It is surprising that linear programming bounds are ever sharp, because
when they are sharp, pair distance information alone suffices to
identify the true ground state.  As one might expect, that is very
rarely the case. To prove stronger bounds, it is natural to try to take
into account triples as well as pairs (i.e., how many times each
triangle of distances occurs between three points), and Schrijver
\cite{Sch} found the right approach.  The constraints are no longer
linear, but rather semidefinite.  Bachoc and Vallentin \cite{BV1}
developed a representation-theoretic explanation and extended the
method from binary codes to spheres, with an approach that applies also
to more general spaces, and this method was further developed by Musin
\cite{M}. These three-point semidefinite programming bounds are one of
the most powerful general tools known for proving bounds in coding
theory. Using them, Bachoc and Vallentin \cite{BV3} determined the
optimal $10$-point code in $S^3$, which was the first new optimality
proof for a spherical code in decades.  They also conjectured that the
bounds are sharp for the optimal $8$-point code in $S^2$ (a square
antiprism, first proved optimal by Sch\"utte and van der Waerden
\cite{SvdW}, so there was less motivation to verify the sharpness in
this case).

\subsection{Our results}

In this paper, we give a new derivation of the semidefinite
programming bounds, which has the advantage of requiring no
explicit calculation beyond what is necessary for the linear
programming bounds.  We then prove semidefinite bounds for
potential energy minimization.  Using three-point bounds, we
prove universal optimality for a seven-point code in $\R\Ps^2$.
It is given by the seven lines through opposite vertices of a
cube and its dual octahedron.  Equivalently, the lines connect
opposite vertices of a rhombic dodecahedron.  In the dual
picture of planes through the origin in $\R^3$, the planes are
parallel to the facets of a cuboctahedron (the dual polyhedron
to the rhombic dodecahedron).

\begin{theorem} \label{theorem:univopt}
The rhombic dodecahedron code is universally optimal in $\R\Ps^2$. It
is globally minimal for energy for each completely monotonic potential
function of squared chordal distance, and it is the unique global
minimum unless the potential function is a linear function.  It is also
the unique optimal seven-point code in $\R\Ps^2$.
\end{theorem}

It is straightforward to check that two-point bounds cannot
prove Theorem~\ref{theorem:univopt}.  It was already known that
this configuration is an optimal projective code (i.e., it
maximizes the minimal distance), as a consequence of the
orthoplex bound of Conway, Hardin, and Sloane \cite{CHS}.
Uniqueness was conjectured in \cite{CHS} but does not follow
from the orthoplex bound; see Appendix~\ref{appendix:orthoplex}
for an explanation.

Note that the uniqueness asserted in Theorem~\ref{theorem:univopt} is
as strong as possible: for linear potential functions, uniqueness fails
because one can rotate the cube and octahedron relative to each other
(so they are no longer in dual position) without changing the energy.
That can be checked directly, but conceptually it holds because both
polyhedra already define projective $1$-designs.  More generally,
$k$-designs automatically minimize potential energy for completely
monotonic polynomials of degree at most $k$.

One noteworthy aspect of the rhombic dodecahedron code is that it is
the first case in which linear or semidefinite programming bounds are
sharp for a code that is not distance regular.  In other words, there
is a distance such that the points in the code do not all have the same
number of neighbors at that distance.  All universal optima known or
conjectured so far have been distance regular, so this example is the
least regular one known.  Theorem~\ref{theorem:univopt} thus extends
the types of configurations that can be analyzed rigorously.

We have searched extensively for other cases in which the three-point
semidefinite programming bounds are sharp (in spheres, real projective
spaces, and complex projective spaces), but the rhombic dodecahedron is
the only new case we have found.  It would be very surprising if it
were the only example in projective space, and we think there must be
others. However, our failure to find them suggests that they are either
large or high-dimensional.  In the process of searching for other
examples, we have formulated several conjectures about semidefinite
programming bounds that would be remarkable if true
(Conjectures~\ref{conjecture:petersen} through~\ref{conjecture:univ}):
in the few sharp cases that are known, the bounds appear to remain
sharp even all the way through a phase transition between two different
structures.

Theorem~\ref{theorem:univopt} provides the last remaining universal
optimum in $\R\Ps^2$:

\begin{theorem} \label{theorem:classify}
The complete list of universally optimal line configurations in $\R^3$
(up to isometry) is as follows:
\begin{enumerate}
\item Up to three orthogonal lines. \label{case:trivial}

\item The four lines through opposite vertices of a cube.
    \label{case:cube}

\item The six lines through opposite vertices of an icosahedron.
    \label{case:icosa}

\item The seven lines through opposite vertices of a rhombic
    dodecahedron. \label{case:new}
\end{enumerate}
\end{theorem}

Case \eqref{case:trivial} is trivial, while cases \eqref{case:cube} and
\eqref{case:icosa} follow from Theorem~8.2 in \cite{CK}.  Case
\eqref{case:new} is Theorem~\ref{theorem:univopt}, and the completeness
of the list is proved in Appendix~\ref{appendix:rp2}.

\subsection{Notation and definitions} \label{subsec:notation}

A \emph{code} is simply a finite subset of a metric space.  It is
\emph{optimal} if it maximizes the minimal distance between points,
given the number of points and the metric space.

An $(n,N,t)$ \emph{spherical code} is a set of $N$ points in the unit
sphere $S^{n-1}$ such that the maximal inner product between distinct
points is at most $t$.  (In other words, the minimal angle between them
is at least $\cos^{-1} t$.)  An \emph{antipodal code} is a spherical
code that is closed under multiplication by $-1$.  Real projective
codes are of course equivalent to antipodal codes, and the rhombic
dodecahedron code corresponds to a $(3,14,1/\sqrt{3})$ antipodal code.
(Note that the vertices of the rhombic dodecahedron do not lie on a
common sphere, so they must be scaled to form this code.)

The \emph{chordal metric} is defined on real, complex, or
quaternionic projective space as follows (in the octonionic
case it is defined using the Frobenius norm on the exceptional
Jordan algebra; see \cite[p.~130]{CK}). If we represent points
in projective space by unit vectors, then the chordal distance
between $x$ and $y$ is $\sqrt{1-|\langle x,y \rangle|^2}$,
where $\langle x,y \rangle$ is the Hermitian inner product.
This metric, first introduced in \cite{CHS}, is equivalent to
the Fubini-Study metric but is in many ways more convenient.

A code in a sphere or projective space is \emph{universally optimal} if
it minimizes the energy $E_f$ for all completely monotonic potential
functions $f$ (compared to all other codes of the same size), with the
metric chosen to be the Euclidean metric on spheres or the chordal
metric on projective spaces. Recall that distance is squared in the
definition of $E_f$. This strengthens the notion of universal
optimality, compared to defining it with unsquared distance, and it
improves the connections with topics such as spherical or projective
designs.  (See \cite[p.~101]{CK}.)

Let $P^n_k(t)$ denote the degree $k$ ultraspherical (i.e., Gegenbauer)
polynomial for $S^{n-1}$, normalized with $P^n_k(1)=1$.  (This is not
the most common notation.) These polynomials are orthogonal with
respect to the measure $(1-t^2)^{(n-3)/2}dt$ on $[-1,1]$.  See
Chapters~6 and~9 of \cite{AAR} for more details.

In the literature, two-point bounds are typically referred to as linear
programming bounds and three-point bounds as semidefinite programming
bounds.  These names are a reasonable reference to the underlying
inequalities, but we propose calling them $k$-{\emph{point bounds}}, which
focuses attention on the most geometrically relevant feature, namely
the number of points being simultaneously considered.

Given a matrix $A$, we write $A \succeq 0$ if $A$ is Hermitian (i.e.,
it equals its own conjugate transpose) and positive semidefinite.  The
inner product on Hermitian matrices is defined by $\langle A,B \rangle
= \trace (A \oversymb{B})$ (i.e., the entry-by-entry Hermitian inner
product), and angle brackets will also denote the usual inner product
on vectors. We will use $J$ to denote a matrix with all entries $1$.

When we deal with points in real projective space $\R\Ps^{n-1}$, we
will represent each point by an arbitrary lift to the double cover
$S^{n-1}$. That introduces sign ambiguities, so we must ensure that all
formulas are invariant under the choice of the lift.  For example, if
we apply a function to the inner product between two points, then it
should be an even function.

\section{Linear and semidefinite programming bounds} \label{sec:lpsdp}

\subsection{Positive-definite kernels} \label{subsec:pdk}

Linear and semidefinite programming bounds are based on the theory of
positive-definite kernels. Let $G$ be a topological group acting
continuously on a topological space $X$.  Call a continuous function $K
\co X \times X \to \C$ a \emph{positive-definite kernel}\footnote{In an
ideal world, we would call $K$ a positive-semidefinite kernel, but
unfortunately this terminology is too well established to be easily
changed.} if for all $N \in \N$ and all $x_1,\dots,x_N \in X$, the $N
\times N$ matrix $(K(x_i,x_j))_{1 \le i,j \le N}$ is Hermitian and
positive semidefinite.  (Of course, saying it is Hermitian simply means
that $K(x,y) = \overline{K(y,x)}$ for all $x,y \in X$.)  We call the
kernel \emph{$G$-invariant} if $K(gx,gy) = K(x,y)$ for all $g \in G$
and $x,y \in X$.

Given any unitary representation $V$ of $G$ with inner product $\langle
\cdot, \cdot \rangle$ and any continuous map $\varphi \co X \to V$ that
is $G$-equivariant (i.e., $\varphi(gx) = g \varphi(x)$ for $g \in G$
and $x \in X$), setting $K(x,y) = \langle \varphi(x), \varphi(y)
\rangle$ defines a $G$-invariant positive-definite kernel on $X$. In
fact, every such kernel arises in this way:

\begin{theorem} \label{theorem:pdk}
For every $G$-invariant positive-definite kernel $K$ on $X$, there
exists a unitary representation $V$ of $G$ and a continuous,
$G$-equivariant map $\varphi \co X \to V$ such that $K(x,y) = \langle
\varphi(x), \varphi(y) \rangle$ for all $x,y \in X$.
\end{theorem}

This theorem is a variant of a characterization of positive-definite
kernels due to Bochner \cite{Bo}.  See, for example, \S5 of Chapter~IV
in \cite{Lang} for a closely related result (with the same idea behind
it).

\begin{proof}
Let $W$ be the complex vector space formally spanned by the points of
$X$, and define a form $\langle \cdot, \cdot \rangle$ on $W$ by setting
$\langle x,y\rangle = K(x,y)$ for $x,y \in X$ and extending linearly in
the first coordinate and conjugate linearly in the second. The action
of $G$ on $X$ extends to an action of $G$ on $W$.

Because $K$ is a $G$-invariant positive-definite kernel, the form
$\langle \cdot, \cdot \rangle$ on $W$ is $G$-invariant and positive
semidefinite.  Let $W_0$ be the subspace of vectors that have norm $0$.
Then the inner product on $W/W_0$ is positive definite, and the
completion $V$ of $W/W_0$ is a Hilbert space.  Let $\varphi \co X \to
V$ be the obvious map (the composition of the trivial embedding of $X$
in $W$, the quotient map by $W_0$, and the embedding in the
completion). Then $\varphi$ is continuous and $G$-equivariant, and
$K(x,y) = \langle \varphi(x), \varphi(y) \rangle$ for all $x,y \in X$.

To complete the proof, we need only verify that $V$ is a unitary
representation of $G$.  In other words, for each $v \in V$, the map $g
\mapsto gv$ must be a continuous function from $G$ to $V$.  By
$G$-invariance it suffices to verify continuity at $g=1$.  Thus, we
wish to show that $gv \to v$ as $g \to 1$.  Furthermore, it suffices to
show this convergence for a dense subset of $V$, and we choose $W/W_0$
as that subset. Therefore, we can assume that
$$
v = \sum_{x \in X} c_x [x],
$$
where $[x]$ denotes the vector in $V$ corresponding to $x \in X$ and
where all but finitely many of the coefficients $c_x$ vanish.  Then
$$
|gv-v|^2 = \sum_{x,y \in X} c_x \overline{c_y} (2K(x,y) - K(x,gy) - K(gx,y)).
$$
By the continuity of $K$ and the action of $G$ on $X$, the finitely
many terms on the right side with nonvanishing coefficients can be made
arbitrarily small by making $g$ close to the identity.  Thus, $V$ is
indeed a unitary representation of $G$.
\end{proof}

If $V$ splits into an orthogonal direct sum $V = V_1 \oplus V_2$ as a
representation of $G$, then each corresponding kernel $K$ splits as $K
= K_1+K_2$, where $K_1$ and $K_2$ correspond to the subrepresentations
$V_1$ and $V_2$.

Suppose that $G$ is compact.  Then by the Peter-Weyl theorem, every
unitary representation of $G$ is an orthogonal direct sum of
finite-dimensional irreducible representations.  Thus, the cone of
$G$-invariant positive-definite kernels is spanned by those arising
from irreducible representations.

When $G$ acts transitively on $X$, we can identify $X$ with $G/H$,
where $H$ is the stabilizer of a point $e \in X$.  Then a
$G$-equivariant map $\varphi \co X \to V$ is completely determined by
$\varphi(e)$ via $\varphi(ge) = g \varphi(e)$, and $\varphi(e)$ can be
any vector in $V$ that is fixed by $H$.

The simplest case is when $(G,H)$ is a Gelfand pair: then the fixed
space $V^H$ has dimension at most $1$ when $V$ is irreducible, so each
irreducible representation is associated with at most one
positive-definite kernel (up to scaling). This occurs, for example,
when $X$ is a sphere, projective space, or Grassmannian and $G$ is the
group of isometries of $X$.

When the fixed space $V^H$ has dimension greater than $1$, the
situation is more complicated.  Then $V$ is associated with several
positive-definite kernels, but in fact there is a richer structure
behind them. For $x_1,x_2 \in X$, define a sesquilinear form
$K_{x_1,x_2}$ on $V^H$ as follows. Let $x_i=g_ie$ with $g_1,g_2 \in G$,
and for $v_1,v_2 \in V^H$ define
$$
K_{x_1,x_2}(v_1,v_2) = \langle g_1v_1, g_2v_2 \rangle.
$$
The key property of these forms is that for all $N \in \N$,
$x_1,\dots,x_N \in X$, and $v_1,\dots,v_N \in V^H$,
$$
\sum_{1 \le i,j \le N} K_{x_i,x_j}(v_i,v_j) \ge 0.
$$
To prove this inequality, note that
$$
\sum_{1 \le i,j \le N} K_{x_i,x_j}(v_i,v_j) =
\sum_{1 \le i,j \le N}
\langle g_i v_i, g_j v_j \rangle =
\left|\sum_{1 \le i \le N} g_i v_i \right|^2.
$$
As a consequence (by taking $v_1 = \dots = v_N$), for all $N \in \N$
and $x_1,\dots,x_N \in X$,
$$
\sum_{1 \le i,j \le N} K_{x_i,x_j} \succeq 0,
$$
where $A \succeq 0$ means that $A$ is positive-semidefinite and
Hermitian.  The $G$-invariant positive-definite kernels are given by
$(x,y) \mapsto K_{x,y}(v,v)$ with $v \in V^H$, but the underlying
structure of $K$ is of broader importance.

In general, define a \emph{matrix-valued positive-definite kernel} to
be a map that takes $x,y \in X$ to a sesquilinear form $K_{x,y}$ on a
fixed complex vector space $W$ (not necessarily finite-dimensional),
with the following three properties:
\begin{enumerate}
\item For all $v,w \in W$, the map $(x,y) \mapsto K_{x,y}(v,w)$ is
    continuous.

\item For all $x,y \in X$ and $v,w \in W$,
$$
K_{x,y}(v,w) = \overline{K_{y,x}(w,v)}.
$$

\item For all $N \in \N$, all $x_1,\dots,x_N \in X$, and all
    $w_1,\dots,w_N \in W$,
$$
\sum_{1 \le i,j \le N}  K_{x_i,x_j}(w_i,w_j)
\ge 0.
$$
\end{enumerate}
The last two properties are equivalent to requiring that
$$
\big(K_{x_i,x_j}(w_i,w_j)\big)_{1 \le i,j \le N} \succeq 0.
$$
We say that $K$ is defined on $X$ and over $W$.  The matrix-valued kernel is
\emph{$G$-invariant} if $K_{gx,gy} = K_{x,y}$ for all $g \in G$ and
$x,y \in X$.  The construction above defines a $G$-invariant
matrix-valued positive-definite kernel over $V^H$ whenever $G$ acts
transitively on $X$.  Note that when $\dim W = 1$, a matrix-valued
positive-definite kernel over $W$ is the same as an ordinary
(scalar-valued) kernel.

When $G$ does not act transitively on $X$, the identification
of the vector space $\Hom_G(X,V)$ of continuous,
$G$-equivariant maps with $V^H$ breaks down. Nevertheless,
essentially the same construction works. One can define a
$G$-invariant matrix-valued positive-definite kernel over
$\Hom_G(X,V)$ by setting
$$
K_{x_1,x_2}(\varphi_1,\varphi_2) =
\langle \varphi_1(x_1), \varphi_2(x_2) \rangle
$$
for $x_1,x_2 \in X$ and $\varphi_1, \varphi_2 \in \Hom_G(X,V)$. This
construction is fully general, aside from changing variables in trivial
ways:

\begin{theorem} \label{theorem:matrixpdk}
For every $G$-invariant matrix-valued positive-definite kernel $K$ on
$X$ and over $W$, there exists a unitary representation $V$ of $G$ and
a linear function $\varphi \co W \to \Hom_G(X,V)$ (written $w \mapsto
\varphi_w$) such that for all $x,y \in X$ and $v,w \in W$,
$$
K_{x,y}(v,w) = \langle \varphi_v(x), \varphi_w(y) \rangle.
$$
\end{theorem}

\begin{proof}
Let $U = W \otimes_\C \C X$, where $\C X$ denotes the complex vector
space formally spanned by the points of $X$.  We define a form $\langle
\cdot, \cdot \rangle$ on $U$ by
$$
\left\langle \sum_i v_i \otimes x_i,
\sum_j w_j \otimes x_j \right\rangle =
\sum_{i,j} K_{x_i,x_j}(v_i,w_j)
$$
for $v_i,w_i \in W$ and $x_i \in X$.  Because $K$ is a matrix-valued
positive-definite kernel, this defines a Hermitian,
positive-semidefinite form on $U$.  Let $U_0$ denote the subspace of
vectors with norm $0$, and let $V$ be the completion of $U/U_0$, so $V$
is a Hilbert space.  Define $\varphi_w(x)$ to be the image of $w\otimes
x$ in $V$.

The Hilbert space $V$ is a unitary representation of $G$, with the
trivial action of $G$ on $W$ and the usual action on $X$.  (Strong
continuity follows as in the proof of Theorem~\ref{theorem:pdk}.)  By
construction,
$$
K_{x,y}(v,w) = \langle \varphi_v(x), \varphi_w(y) \rangle,
$$
which completes the proof.
\end{proof}

When $X = S^{n-1}$ and $G$ is the stabilizer in $O(n)$ of a
point in $X$, the matrices constructed in Theorem~3.1 of
\cite{BV1} (by a different method from that used here) define
matrix-valued positive-definite kernels, and we were led to our
construction in an attempt to develop a more abstract approach
to Bachoc and Vallentin's discovery.

\subsection{Linear and semidefinite programming bounds}

Given the machinery of positive-definite kernels, it is easy to write
down linear and semidefinite programming bounds for energy
minimization.  We will write $X^2/G$ for the set of orbits of $G$
acting on pairs, and $[x,y]$ will denote the orbit of the pair $(x,y)$.

The linear programming bounds involve linear constraints on the
\emph{$G$-invariant pair distribution} of a code. Given a finite subset
$\mathcal{C} \subseteq X$, define
$$
A_{x,y} = |\{(u,v) \in \mathcal{C}^2 : [u,v] = [x,y]\}|.
$$
These numbers satisfy some obvious linear constraints.  They are all
nonnegative, and their sum over the set $X^2/G$ of orbits is
$|\mathcal{C}|^2$. Furthermore, $A_{x,x} = |\mathcal{C}|$ if $G$ acts
transitively on $X$, and even if that is not the case, the diagonal
terms $A_{x,x}$ are the orbit sizes of $G$ acting on $X$.

In addition to these obvious constraints, each $G$-invariant
positive-definite kernel $K \co X \times X \to \C$ gives a more subtle
constraint.  Specifically,
$$
\sum_{[x,y] \in X^2/G} A_{x,y} K(x,y) =
\sum_{x,y \in \mathcal{C}} K(x,y) \ge 0.
$$
Of course, the sum over $[x,y] \in X^2/G$ means a sum over
representatives of the orbits.  Positive-definite kernels are
remarkable because they have this property despite typically being
negative at many points. Pfender \cite{P} discovered that they are not
the only such functions, and one can occasionally improve the linear
programming bounds by incorporating his functions. Unfortunately the
improvement seems to be small in general, and the method becomes much
less systematic (because the representation-theoretic context is lost),
so we will not use Pfender's approach here.

Given a $G$-invariant potential function $f \co X \times X \to
\R$, define the energy of $\mathcal{C}$ by
$$
\frac{1}{2}
\sum_{\shortstack[c]{$\scriptstyle x,y \in \mathcal{C}$\\
$\scriptstyle x \ne y$}}
f(x,y) = \frac{1}{2}
\sum_{\shortstack[c]{$\scriptstyle [x,y] \in X^2/G$\\
$\scriptstyle x \ne y$}}
A_{x,y} f(x,y).
$$
Potential energy is a linear functional of the pair distribution $A$,
so by solving a linear programming problem (possibly infinite dimensional) one can optimize it subject to the linear constraints
described above.  To prove a lower bound for energy, valid for
arbitrary sets of $|\mathcal{C}|$ points in $X$, one need only find a
feasible point in the dual linear program.

When $X = G/H$ and $(G,H)$ is a Gelfand pair, linear programming bounds
express all systematically available information about the pair
distribution.  In more general cases, linear programming bounds can be
generalized to two-point semidefinite programming bounds.  They work
the same way, except they use the matrix-valued positive-definite
kernels $K_{x,y}$ developed in Subsection~\ref{subsec:pdk}.  For each
such kernel,
$$
\sum_{[x,y] \in X^2/G} A_{x,y} K_{x,y} \succeq 0,
$$
which is a semidefinite constraint on $A$.  Thus, optimizing energy
subject to these constraints becomes a semidefinite programming
problem.

\subsection{Three-point bounds and beyond}

The most important application of semidefinite programming bounds is to
prove three-point bounds (as in \cite{BV1}).  Fix a point $e \in X$ and
let $H = \Stab_G(e)$ be its stabilizer.  The semidefinite programming
bounds give constraints on $H$-invariant pair distributions on $X$, and
symmetrizing $e$ with the other two points turns them into constraints
on $G$-invariant triple distributions.  One can then attempt to
optimize energy (or other quantities) subject to these constraints.

Of course, we are not limited to using just three points, and one can
prove $k$-point bounds in the same way.  (Musin \cite{M} was the first
to formulate these bounds for spherical codes.) However, as $k$
increases the bounds become increasingly difficult to compute with. The
difficulty is that $k$-point distributions are functions on $X^k/G$, so
$k$-point bounds involve optimization over functions on this space.
When $X^k/G$ is one-dimensional, the optimization problem is usually
doable, but it rapidly becomes intractable as the dimension increases.

For a concrete example, suppose $X = S^{n-1}$ and $G = O(n)$.  Then
elements of $X^k/G$ are determined by their pairwise distances, so
$\dim X^k/G = \binom{k}{2}$.  For $k=3$, functions of three variables
are barely tractable, but for $k=4$, functions of six variables are too
complicated: the space of polynomials of degree at most $m$ in six
variables has dimension $\binom{m+6}{6}$, which grows too rapidly to
allow for extensive computations.

The situation can be much worse for other spaces.  For example, a pair
of points in $\C\Ps^{n-1}$ is determined up to isometries by the
distance between them, but for triples of points there is a fourth
parameter, namely a complex phase change.  (Two unit vectors in $\C^n$
can easily be phase shifted so that their inner product is real, but
for three vectors that is generally impossible.)  The Grassmannian of
$k$-dimensional subspaces in $\R^n$ is even worse, since $\min(k,n-k)$
parameters are required to determine a pair up to isometries (see, for
example, \cite{B}).

\subsection{Explicit computations}

Of course, to apply any of these bounds in practice, one must
carry out the representation-theoretic computations explicitly.
Fortunately, in the cases of interest in this paper, no more
calculations are needed to set up $k$-point bounds than
two-point bounds.

Recall that for the sphere $S^{n-1}$, it is a theorem of Schoenberg
\cite{Scho} that the cone of $O(n)$-invariant positive-definite kernels
is spanned by the functions $(x,y) \mapsto P_k^n(\langle x,y \rangle)$.
Here, $P_k^n$ denotes the $k$-th degree Gegenbauer polynomial for
$S^{n-1}$. See Subsection~2.2 of \cite{CK} for a brief account of these
calculations.  In terms of Theorem~\ref{theorem:pdk}, these
positive-definite kernels correspond to the irreducible representations
of $O(n)$ that contain a nonzero vector fixed by $O(n-1)$ (and the
other irreducible representations do not have corresponding kernels).

Once this characterization of positive-definite kernels is known, one
can extend it to $k$-point bounds without needing any additional
representation theory or information about special functions.  For
example, three-point bounds can be handled as follows.  This gives a
new proof of Corollary~3.5 in \cite{BV1}.

\begin{theorem} \label{theorem:explicit}
Let $H$ be the stabilizer in $O(n)$ of a point $e$ in $S^{n-1}$. For
each $k \ge 0$, there is an $H$-invariant matrix-valued
positive-definite kernel on $S^{n-1}$ that takes $x_1,x_2 \in S^{n-1}$
to the infinite matrix whose $(i_1,i_2)$-entry (indexed starting with
$0$) is
$$
u_1^{i_1} u_2^{i_2} \big((1-u_1^2)(1-u_2^2)\big)^{k/2}
P^{n-1}_k \left(\frac{t-u_1u_2}{\sqrt{(1-u_1^2)(1-u_2^2)}}\right),
$$
where $u_j = \langle e, x_j \rangle$ and $t = \langle x_1,x_2 \rangle$.
\end{theorem}

Of course, for numerical computations one uses finite submatrices of
these infinite matrices.

\begin{proof}
Let $X = S^{n-1}$, with $e$ and $H$ as in
Theorem~\ref{theorem:explicit}.  We must compute matrix-valued
positive-definite kernels on the vector space $\Hom_H(X,V)$, where $V$
is an irreducible unitary representation of $H$. Of course,
$\Hom_H(X,V)$ is infinite dimensional, but we will choose a basis for a
dense subset.

Let
$$
Y = \{ y \in S^{n-1} : \langle e,y \rangle = 0\}
$$
be the equator relative to $e$.  To specify a map $\varphi \in
\Hom_H(X,V)$, we just need to specify its restriction to the equator
and to all the parallel slices of the sphere.  Fortunately, that is
relatively manageable.  As in Subsection~\ref{subsec:pdk}, $\Hom_H(Y,V)
\cong V^H$ because $H$ acts transitively on $Y$, and the Gelfand pair
property implies that $V^H$ is at most one-dimensional; the same is
true for all the parallel slices. Thus, the map $\varphi$ is determined
by specifying a scaling factor on each slice, which is a function of
one variable.

Suppose $\dim V^H = 1$, and choose an element $\tilde \varphi \in
\Hom_H(Y,V)$ that is not identically zero.  We will use $\tilde
\varphi$ to define $\varphi$ on the equator $Y$.  To complete the
determination of $\varphi$, we must specify the scaling factor on each
parallel slice. We will use scaling factors that are polynomials in the
inner product with $e$ (up to multiplication by some fixed function $f$
to be specified later). Specifically, for $-1 \le u \le 1$ and $y \in
Y$, define
$$
\varphi_i\big(u e + \sqrt{1-u^2} y\big) = u^i f(u) \tilde \varphi(y).
$$
The function $\varphi_i$ is $H$-equivariant because $\tilde \varphi$
is, and as $i$ varies these functions are dense in $\Hom_H(X,V)$.

To compute the matrix-valued positive-definite kernel, we need to
compute $\langle \varphi_{i_1}(x_1), \varphi_{i_2}(x_2)\rangle$ for
$x_1,x_2 \in X$.  If we write $x_j = u_j e + \sqrt{1-u_j^2}\, y_j$,
then
$$
\langle \varphi_{i_1}(x_1), \varphi_{i_2}(x_2)\rangle =
u_1^{i_1} u_2^{i_2} f(u_1) f(u_2)\,
\langle \tilde \varphi (y_1), \tilde \varphi(y_2) \rangle.
$$
The inner product $\langle \tilde \varphi (y_1), \tilde
\varphi(y_2) \rangle$ can be computed using Schoenberg's theorem,
because $Y$ is a sphere and $H$ is the full isometry group of $Y$.
Specifically, there is some $k$ (depending on $V$) such that
$$
\langle \tilde \varphi (y_1), \tilde
\varphi(y_2) \rangle = P^{n-1}_k(\langle y_1,y_2 \rangle)
$$
if $\tilde \varphi$ is rescaled appropriately, which we can assume.

If we let $t = \langle x_1,x_2 \rangle$, then
$$
\langle y_1,y_2 \rangle = \frac{t-u_1u_2}{\sqrt{(1-u_1^2)(1-u_2^2)}}.
$$
It follows that the $(i_1,i_2)$-entry $\langle \varphi_{i_1}(x_1),
\varphi_{i_2}(x_2)\rangle$ of the matrix-valued positive-definite
kernel is
$$
u_1^{i_1} u_2^{i_2} f(u_1) f(u_2) \,
P^{n-1}_k \left(\frac{t-u_1u_2}{\sqrt{(1-u_1^2)(1-u_2^2)}}\right).
$$

It is convenient to choose $f(u) = (1-u^2)^{k/2}$.  Then
$$
\langle \varphi_{i_1}(x_1), \varphi_{i_2}(x_2)\rangle =
u_1^{i_1} u_2^{i_2} \big((1-u_1^2)(1-u_2^2)\big)^{k/2}
P^{n-1}_k \left(\frac{t-u_1u_2}{\sqrt{(1-u_1^2)(1-u_2^2)}}\right).
$$
This choice has the advantage that the right side is a polynomial in
$u_1$, $u_2$, and $t$ (because $P_k^{n-1}$ is an even or odd function,
according as $k$ is even or odd).
\end{proof}

The same approach works straightforwardly for $k$-point bounds:
choosing $H$ to be the stabilizer of several points recovers the
results of Musin \cite{M}.  It also works for projective spaces.  As in
the case of spheres, the calculations for an $n$-dimensional projective
space with respect to the stabilizer of a point reduce immediately to
those for an $(n-1)$-dimensional space with respect to the full group.

The calculations in real projective space are almost the same as those
in the sphere.  If we lift points arbitrarily to the sphere (as
discussed in Subsection~\ref{subsec:notation}), then we just need to
avoid any sign ambiguity. Specifically, in terms of the three inner
products $u_1$, $u_2$, and $t$ from Theorem~\ref{theorem:explicit}, we
want only terms that are invariant under changing the signs of two of
the three variables.  That means we take only the entries for which
$i_1 \equiv i_2 \equiv k \pmod{2}$. This submatrix defines a
matrix-valued positive-definite kernel on real projective space.

\section{Three-point bounds for energy minimization}

In this section, we write down three-point bounds for energy
minimization in spheres or real projective spaces.  We begin with
spheres, after which we can easily adapt the answer to real projective
spaces. All of the required theory is in Section~\ref{sec:lpsdp}, but
there are a number of details that must be worked out carefully.

It will be convenient to write potential functions in terms of the
inner product.  Given a potential function $f \co [-1,1) \to \R$,
define
$$
\widetilde{E}_f(\mathcal{C}) = \frac{1}{2}
\sum_{\shortstack[c]{$\scriptstyle x,y \in \mathcal{C}$\\
$\scriptstyle x \ne y$}}
f\big(\langle x,y \rangle\big).
$$
Because $|x-y|^2 = 2 - 2 \langle x,y \rangle$, if we set $f(t) =
g(2-2t)$, then $\widetilde{E}_f(\mathcal{C}) = E_g(\mathcal{C})$.  The
function $f$ is absolutely monotonic (i.e., it is infinitely
differentiable and all of its derivatives are nonnegative) if and only
if $g$ is completely monotonic.  Thus, $\mathcal{C}$ is universally
optimal if and only if it minimizes $\widetilde{E}_f$ for all
absolutely monotonic $f$.

Given a configuration $\mathcal{C} \subseteq S^{n-1}$ with
$|\mathcal{C}|=N>2$, define the corresponding triple distribution by
$$
A_{u,v,t} = |\{(x,y,z) \in \mathcal{C}^3 : \langle x,y \rangle = u,
\langle y,z \rangle = v, \langle z,x \rangle = t\}|.
$$
This function counts ordered triples of points modulo the action of
$O(n)$. Of course, $A_{u,v,t} \ge 0$, and $A_{u,v,t} = 0$ unless $-1
\le u,v,t \le 1$ and
$$
1+2uvt-u^2-v^2-t^2 \ge 0.
$$
In other words, the $3 \times 3$ Gram matrix
$$
\begin{pmatrix}
1 & u & v\\
u & 1 & t\\
v & t & 1
\end{pmatrix}
$$
must be positive semidefinite.  Furthermore, $A_{u,v,t}$ is symmetric
in $u$, $v$, and $t$.  It satisfies the identities $A_{1,1,1} = N$,
$$
\sum_{u} A_{u,u,1} = N^2,
$$
$$
\sum_{u,v,t} A_{u,v,t} = N^3,
$$
and
$$
\sum_{v,t} A_{u,v,t} = N A_{u,u,1}.
$$

Define
$$
D = \{(u,v,t) : -1 \le u,v,t < 1 \textup{ and } 1+2uvt-u^2-v^2-t^2 \ge 0\}.
$$
Because the potential function is not necessarily defined at inner
product $1$, it will be important to work with sums over $D$. For
example,
\begin{align*}
\sum_{(u,v,t) \in D} A_{u,v,t} &= \sum_{u,v,t} A_{u,v,t} - 3
\sum_u A_{u,u,1} + 2 A_{1,1,1}\\
&= N^3 - 3N^2 + 2N\\
&= N(N-1)(N-2).
\end{align*}
Equivalently, summing over $D$ counts ordered triples of distinct
points.

We can express $\widetilde{E}_f(\mathcal{C})$ as a sum over $D$ by
writing
\begin{align*}
\widetilde{E}_f(\mathcal{C}) &= \frac{1}{2}\sum_{u < 1} A_{u,u,1} f(u)\\
& = \frac{1}{2N} \sum_{u<1, v, t} A_{u,v,t} f(u) \\
& = \frac{1}{2N} \left(2 \sum_{u < 1} A_{u,u,1} f(u) +
\sum_{(u,v,t) \in D} A_{u,v,t} f(u) \right)\\
& = \frac{1}{2N} \left(4 \widetilde{E}_f(\mathcal{C}) +
\sum_{(u,v,t) \in D} A_{u,v,t} f(u) \right).
\end{align*}
It follows that
\begin{equation} \label{eqn:Eformula}
\widetilde{E}_f(\mathcal{C}) =
\frac{1}{6(N-2)} \sum_{(u,v,t) \in D} A_{u,v,t} \big(f(u) + f(v) + f(t)\big).
\end{equation}

Let $S_k^n(u,v,t)$ be the infinite matrix whose $(i,j)$-entry (indexed
starting with $0$) is the symmetrization of
$$
u^{i} v^{j} \big((1-u^2)(1-v^2)\big)^{k/2}
P^{n-1}_k \left(\frac{t-uv}{\sqrt{(1-u^2)(1-v^2)}}\right)
$$
in $u$, $v$, and $t$ (i.e., the average over all permutations of the
variables).  The sum of this matrix over the code is positive
semidefinite by Theorem~\ref{theorem:explicit}; in terms of
$A_{u,v,t}$,
$$
\sum_{u,v,t} A_{u,v,t} S_k^n(u,v,t) \succeq 0.
$$
If we break the sum up according to which variables equal $1$, we find
that
$$
\sum_{(u,v,t) \in D} A_{u,v,t} S_k^n(u,v,t) +
3 \sum_{u<1} A_{u,u,1} S^n_k(u,u,1) + N S^n_k(1,1,1) \succeq 0.
$$
(It might appear that $S^n_k(u,v,t)$ is undefined when one of the
variables equals $1$, but in fact all its entries are polynomials in
$u$, $v$, and $t$.) The same trick as we used for $\widetilde{E}_f$
shows that $\sum_{u<1} A_{u,u,1} S^n_k(u,u,1)$ equals
$$
\frac{1}{3(N-2)}
\sum_{(u,v,t) \in D} A_{u,v,t} \big(S^n_k(u,u,1) + S^n_k(v,v,1) + S^n_k(t,t,1)\big).
$$
(In fact, this is simply \eqref{eqn:Eformula} with $f(u)$ replaced with
$2 S^n_k(u,u,1)$.)  Thus, if we define
$$
T^n_k(u,v,t) = (N-2) S^n_k(u,v,t) + S^n_k(u,u,1) + S^n_k(v,v,1) + S^n_k(t,t,1),
$$
then
\begin{equation*}
N(N-2)S^n_k(1,1,1) + \sum_{(u,v,t) \in D} A_{u,v,t} T^n_k(u,v,t) \succeq 0.
\displaybreak\end{equation*}
Furthermore, note that $S^n_k(1,1,1)$ is the zero matrix unless $k=0$,
in which case all of its entries are $1$.  If we let $J$ denote the all
$1$ matrix, then our inequality becomes
$$
N(N-2)\delta_{k,0} J + \sum_{(u,v,t) \in D} A_{u,v,t} T^n_k(u,v,t) \succeq 0,
$$
where $\delta$ denotes the Kronecker delta function.

\begin{theorem} \label{theorem:primal}
The minimal value of $\widetilde{E}_f$ for $N>2$ points in $S^{n-1}$ is
greater than or equal to the optimum of the following semidefinite
program: minimize
$$
\frac{1}{6(N-2)} \sum_{(u,v,t) \in D} A_{u,v,t} \big(f(u) + f(v) + f(t)\big)
$$
over all choices of $A_{u,v,t} \ge 0$ for $(u,v,t) \in D$, subject to
$$
\sum_{(u,v,t) \in D} A_{u,v,t} =
N(N-1)(N-2)
$$
and
$$
N(N-2)\delta_{k,0} J + \sum_{(u,v,t) \in D} A_{u,v,t} T^n_k(u,v,t) \succeq 0
$$
for $k \ge 0$.
\end{theorem}

To compute a lower bound for energy, we will use the dual semidefinite
program.  Suppose we define a function $H$ on $D$ by
$$
H(u,v,t) = c + \sum_{k \ge 0} \langle F_k, T_k^n(u,v,t) \rangle,
$$
where $c$ is a constant, $F_k$ denotes an infinite symmetric
matrix, and the inner product on symmetric matrices is the
trace of the product.  We assume only finitely many entries of
$F_k$ are nonzero, so the inner product is well defined.

\begin{theorem} \label{theorem:dual}
With the notation established above, if $F_k \succeq 0$ for all $k$ and
$$
H(u,v,t) \le \frac{f(u)+f(v)+f(t)}{3}
$$
for all $(u,v,t) \in D$, then the minimal value of $\widetilde{E}_f$
for $N$ points in $S^{n-1}$ is at least
$$
\frac{N}{2}\big((N-1)c - \langle F_0, J \rangle\big).
$$
\end{theorem}

\begin{proof}
For any $N$-point code $\mathcal{C}$ with triple distribution $A$,
$$
\sum_{(u,v,t) \in D} A_{u,v,t} H(u,v,t) \le 2(N-2)\widetilde{E}_f(\mathcal{C})
$$
by \eqref{eqn:Eformula}. On the other hand, the inner product
of two positive semidefinite matrices is nonnegative, from
which it follows that
$$
N(N-2)\langle F_k, J\rangle \delta_{k,0}
+ \sum_{(u,v,t) \in D} A_{u,v,t} \langle F_k,
T^n_k(u,v,t) \rangle \ge 0.
$$
Summing over $k$ yields
$$
N(N-2)\langle F_0, J\rangle + \sum_{(u,v,t) \in D} A_{u,v,t}
\big(H(u,v,t)-c\big) \ge 0,
$$
and combining these inequalities with $\sum_{(u,v,t) \in D} A_{u,v,t} =
N(N-1)(N-2)$ completes the proof.
\end{proof}

As discussed at the end of Section~\ref{sec:lpsdp}, the real
projective case is almost identical.  If $f$ is an even
function, then we can lift points arbitrarily to $S^{n-1}$
without introducing any ambiguity in $\widetilde{E}_f$.  It
will prove convenient to use a third variant $\widehat{E}$ of
the notation for energy (in addition to $E$ and
$\widetilde{E}$); that may seem unnecessary, but $E$ has the
clearest connection to physics, $\widetilde{E}$ is the most
convenient for spheres, and $\widehat{E}$ is the most
convenient for real projective space. Given a potential
function $f \co [0,1) \to \R$, define
$$
\widehat{E}_f(\mathcal{C}) = \frac{1}{2}
\sum_{\shortstack[c]{$\scriptstyle x,y \in \mathcal{C}$\\
$\scriptstyle x \ne y$}}
f\big(\langle x,y \rangle^2\big).
$$
Here $x$ and $y$ represent lifts to $S^{n-1}$ of points in
$\R\Ps^{n-1}$.  Note that each point in $\R\Ps^{n-1}$ has two
lifts, but we only include one of them in the sum (chosen
arbitrarily).

Let $\widehat{S}_k^n$ be the submatrix of $S_k^n$ indexed by numbers
with the same parity as $k$, and define
$$
\widehat{T}^n_k(u,v,t) = (N-2) \widehat{S}^n_k(u,v,t)
+ \widehat{S}^n_k(u,u,1) + \widehat{S}^n_k(v,v,1) + \widehat{S}^n_k(t,t,1).
$$
Then the three-point energy bounds for $\R\Ps^{n-1}$ are exactly the
same as those for $S^{n-1}$, except with $\widehat{S}$ and
$\widehat{T}$ replacing $S$ and $T$:

\begin{theorem} \label{theorem:dualproj}
If
$$
H(u,v,t) = c + \sum_{k \ge 0} \langle F_k, \widehat{T}_k^n(u,v,t) \rangle,
$$
with $F_k \succeq 0$ for all $k$ and
$$
H(u,v,t) \le \frac{f\big(u^2\big)+f\big(v^2\big)+f\big(t^2\big)}{3}
$$
for all $(u,v,t) \in D$, then the minimal value of $\widehat{E}_f$ for
$N$ points in $\R\Ps^{n-1}$ is at least
$$
\frac{N}{2}\big((N-1)c - \langle F_0, J
\rangle\big).
$$
\end{theorem}

The choice of $c$ and $F_0,F_1,\dots$ in
Theorems~\ref{theorem:dual} and \ref{theorem:dualproj} can be
optimized using semidefinite programming. First, assume $F_k=0$
for $k$ beyond some bound, and that for all $k$ the $(i,j)$-entry
of $F_k$ vanishes if $i$ and $j$ are sufficiently large, so
that only finitely many variables need to be considered. Assume
also that $f$ is a polynomial (although that assumption can be
relaxed, at the cost of additional complications). To express
the constraint that
$$
H(u,v,t) \le \frac{f\big(u^2\big)+f\big(v^2\big)+f\big(t^2\big)}{3}
$$
for $(u,v,t) \in D$, we will use a sum of squares representation due to
Putinar.  Unfortunately, there is a nontrivial error in the paper
\cite{Pu} (see \cite[p.~98]{Ma}), so we must be careful here.  The
space $D$ is defined by the constraints $1-u^2 \ge 0$, $1-v^2 \ge 0$,
$1-t^2 \ge 0$, and $1+2uvt-u^2-v^2-t^2 \ge 0$, and this representation
of $D$ is stably compact (i.e., it defines a compact set even if the
coefficients are perturbed); furthermore, only the last constraint has
odd degree. It follows from Corollary~7.2.5 in \cite{Ma} that if a
polynomial is strictly positive on $D$, then it can be expressed in the
form
\begin{align*}
 G_0(u,v,t) &+ G_1(u,v,t)(1-u^2) + G_2(u,v,t)(1-v^2) + G_3(u,v,t)(1-t^2)\\
&
+ G_4(u,v,t)(1+2uvt-u^2-v^2-t^2),
\end{align*}
where $G_0(u,v,t),\dots,G_4(u,v,t)$ are sums of squares of polynomials
in $u,v,t$.  A polynomial is a sum of squares if and only if it is of
the form $\langle z(u,v,t), M z(u,v,t)\rangle$, where $M$ is a
symmetric, positive-semidefinite matrix and $z(u,v,t)$ is a vector
whose entries are monomials in $u,v,t$; thus, being a sum of squares is
a semidefinite condition. If we apply this approach to
$$
\frac{f\big(u^2\big)+f\big(v^2\big)+f\big(t^2\big)}{3} - H(u,v,t),
$$
then we get semidefinite programs that are guaranteed to come
arbitrarily close to the true optimum in Theorem~\ref{theorem:dual}. Of
course, we hope they will actually achieve the true optimum, and in
practice that occurs.

\section{Proof of universal optimality}

In this section we prove Theorem~\ref{theorem:univopt}.  There are two
fundamental difficulties.  One is that, although we can solve a
semidefinite program numerically to obtain a bound for any given
potential function, the solutions are very cumbersome and it is not
easy to produce a rigorous, exact solution.  Bachoc and Vallentin dealt
with this difficulty in \cite{BV3}, but their problem is substantially
simpler and they employed ad hoc techniques.  We will develop a more
systematic approach.

The second difficulty is that we must
prove optimality for infinitely many different potential functions,
namely all completely monotonic functions.  Of course we only need to
deal with the extreme rays in this cone, but there are infinitely many
of them as well.  To address this difficulty, we will construct a
finite set of functions (not all completely monotonic) such that
optimality for all of them implies universal optimality.  In other
words, we will replace the cone of completely monotonic functions with
a larger cone that has only finitely many extreme rays, so we will
prove a slightly stronger result than universal optimality.

Recall that
$$
\widehat{E}_f(\mathcal{C}) = \frac{1}{2}
\sum_{\shortstack[c]{$\scriptstyle x,y \in \mathcal{C}$\\
$\scriptstyle x \ne y$}}
f\big(\langle x,y \rangle^2\big),
$$
where $x$ and $y$ represent lifts to $S^{n-1}$ of points in
$\R\Ps^{n-1}$ (with only one lift of each point being used). One
advantage of squaring the inner product is that it becomes invariant
under sign changes, but it also relates well to the chordal distance.
The squared chordal distance between $x$ and $y$ is $1-\langle x,y
\rangle^2$, so $\widehat{E}_f(\mathcal{C}) = E_g(\mathcal{C})$ with
$g(t) = f(1-t)$. Thus, $\mathcal{C}$ is universally optimal in
$\R\Ps^{n-1}$ if and only if it minimizes $\widehat{E}_f$ for each
absolutely monotonic function $f$. The cone of absolutely monotonic
functions on $[0,1)$ is spanned by the monomials $f(t) = t^i$ (see
Theorem~9b in \cite[p.~154]{W}), so it suffices to prove optimality for
these monomials.

To reduce proving universal optimality to a finite problem, we
will apply Hermite interpolation.  Recall that given a
nonempty, finite multiset $T$ of points in $\R$ (with the
multiplicity of $t \in T$ denoted $\mult_T(t)$), the Hermite
interpolation $\Hermite_T(f)$ of a function $f$ is the unique
polynomial of degree less than
\begin{equation*}
\sum_{t \in T} \mult_T(t) \displaybreak
\end{equation*}
that agrees with $f$ to order $\mult_T(t)$ at each $t \in T$.  (I.e.,
$f^{(i)}(t) = (\Hermite_T (f))^{(i)}(t)$ for all $i < \mult_T(t)$.) See
Subsection~2.1 of \cite{CK} for background on Hermite interpolation.
The following observation of Yudin will be crucial:

\begin{lemma}[Yudin \cite{Y}] \label{lemma:yudin}
Let $T$ be a finite, nonempty multisubset of an interval $I$ such that
each point in $T$ has even multiplicity, except for the endpoints of
$I$ (which are allowed to have even or odd multiplicity). For each
absolutely monotonic function $f \co I \to \R$,
$$
\Hermite_T (f)(t) \le f(t)
$$
for all $t \in I$.
\end{lemma}

Lemma~\ref{lemma:yudin} follows from the remainder formula for Hermite
interpolation (see, for example, Lemmas~2.1 and~5.1 in \cite{CK} for a
proof).

\begin{lemma} \label{lemma:induction}
Let $T = \{t_1,\dots,t_M\}$ be a nonempty multisubset of an interval
$I$ (written with $t_i$ repeated according to its multiplicity). If $f
\co I \to \R$ is absolutely monotonic, then $\Hermite_T(f)$ is a
nonnegative linear combination of the partial products
$$
t \mapsto \prod_{i=1}^m (t-t_i)
$$
for $0 \le m < M$.  (Of course, the $m=0$ case is the constant function
$1$.)
\end{lemma}

What Lemma~\ref{lemma:induction} says depends on the ordering of
$t_1,\dots,t_M$, but it is true for every ordering.

\begin{proof}
We prove Lemma~\ref{lemma:induction} by induction on $M$.  For $M=1$,
$\Hermite_T(f)$ is the constant function $f(t_1)$ and the lemma is
trivial.  Otherwise, let $T' = \{t_2,\dots,t_{M}\}$.  Then
\begin{equation} \label{eqn:Hermiteinduct}
(\Hermite_T(f))(t) = f(t_1) + (t-t_1) (\Hermite_{T'}(g))(t),
\end{equation}
where $g(t) = (f(t)-f(t_1))/(t-t_1)$ (and $g(t_1) = f'(t_1)$).  The
function $g$ is absolutely monotonic on $I$, by Proposition~2.2 in
\cite{CK}; alternatively, that can be seen directly via
$$
\frac{d^k}{dt^k} \left(\frac{f(t)-f(t_1)}{t-t_1}\right) =
\frac{\int_{t_1}^t f^{(k+1)}(u) (u-t_1)^k \, du}{(t-t_1)^{k+1}},
$$
which is the fundamental theorem of calculus for $k=0$ and can be
proved by induction (or by using a Taylor series expansion for
$(f(t)-f(t_1))/(t-t_1)$ about $t_1$). Now applying the lemma
inductively to $\Hermite_{T'}(g)$ and using \eqref{eqn:Hermiteinduct}
completes the proof.
\end{proof}

Combining Lemmas~\ref{lemma:yudin} and~\ref{lemma:induction} reduces
proving universal optimality to a finite number of cases, as follows.

\begin{corollary} \label{corollary:reduce}
Let $\mathcal{C}$ be a finite subset of $\R\Ps^{n-1}$ (represented via
arbitrary lifts to $S^{n-1}$), and let $T = \{t_1, \dots, t_M\}$ be any
finite multisubset of $[0,1)$ (written with multiplicities) such that
all elements other than $0$ have even multiplicity and
$$
\{\langle x,y \rangle^2 : x,y \in \mathcal{C},\, x \ne y\} \subseteq T. 
$$
If
$\mathcal{C}$ minimizes $\widehat{E}_{f_m}$ for the potential functions
$$
f_m(t) = \prod_{i=1}^m (t-t_i)
$$
with $0 \le m < M$, then $\mathcal{C}$ is universally optimal in
$\R\Ps^{n-1}$.
\end{corollary}

\begin{proof}
Let $f \co [0,1) \to \R$ be absolutely monotonic. By
Lemma~\ref{lemma:yudin}, $f \ge \Hermite_T(f)$ on $[0,1)$, and by
Lemma~\ref{lemma:induction}, there are nonnegative coefficients
$\lambda_0,\dots,\lambda_{M-1}$ such that
$$
\Hermite_T(f) = \lambda_0 f_0 + \dots + \lambda_{M-1} f_{M-1}.
$$
Thus, for any configuration $\mathcal{D}$ with $|\mathcal{D}| =
|\mathcal{C}|$,
\begin{align*}
\widehat{E}_f(\mathcal{D}) &\ge \widehat{E}_{\Hermite_T(f)}(\mathcal{D})\\
&= \lambda_0 \widehat{E}_{f_0}(\mathcal{D}) + \dots + \lambda_{M-1}
\widehat{E}_{f_{M-1}}(\mathcal{D})\\
&\ge \lambda_0 \widehat{E}_{f_0}(\mathcal{C}) + \dots + \lambda_{M-1}
\widehat{E}_{f_{M-1}}(\mathcal{C})\\
&= \widehat{E}_{\Hermite_T(f)}(\mathcal{C})\\
&= \widehat{E}_{f}(\mathcal{C}),
\end{align*}
where the last equation holds because $f = \Hermite_T(f)$ at each
squared inner product between distinct points in $\mathcal{C}$.
\end{proof}

Unfortunately, even if $\mathcal{C}$ is universally optimal, there is
no guarantee that it is optimal for the potential functions constructed
in Corollary~\ref{corollary:reduce}.  That seems to depend on luck and
the proper choice of $T$.

The proof of the main theorem in \cite{CK} can be recast in this
framework, which simplifies the argument given there: specifically, it
replaces the use of conductivity.  As a side benefit, this approach
allows us to give a substantially simpler proof of the universal
optimality of the regular $600$-cell in $S^3$ than was given in
\cite{CK}.  (That is the one case that was not proved by a conceptual
argument, but rather by somewhat complicated calculations.) See
Appendix~\ref{appendix:600cell} for more details.

For the rhombic dodecahedron code in $\R\Ps^2$, the squared inner
products are $0$, $1/9$, and $1/3$.  If we take $T = \{0,0,0,1/9, 1/9,
1/3, 1/3\}$ in Corollary~\ref{corollary:reduce}, we find that it
suffices to prove optimality for the potential functions $f(t) = 1$,
$t$, $t^2$, $t^3$, $t^3(t-1/9)$, $t^3(t-1/9)^2$, and
$t^3(t-1/9)^2(t-1/3)$.  Of course the constant function is trivial, and
two-point bounds are sharp for $f(t)=t$ (the rhombic dodecahedron code
is a projective $1$-design). Thus, only five cases remain.

The triple repetition of $0$ in $T$ is not essential for the proof, but
it is helpful.  Our approach fails if we take $\mult_T(0)=1$; with
$\mult_T(0)=2$, it works but the numerical calculations are more
cumbersome.

For the rhombic dodecahedron code $\mathcal{C}$,
$$
\widehat{E}_f(\mathcal{C}) = 3f(0) + 6f(1/9) + 12f(1/3).
$$
Thus, we wish to prove a lower bound of $3f(0) + 6f(1/9) + 12f(1/3)$
for energy with respect to each of the potential functions $f(t) =
t^2$, $t^3$, $t^3(t-1/9)$, $t^3(t-1/9)^2$, and $t^3(t-1/9)^2(t-1/3)$.
In fact, in each case we prove something stronger, namely that the same
lower bound holds not just for the potential function $f(t)$ but also
for
$$
f_0(t) = f(t) - \frac{t^3(t-1/9)^2(t-1/3)^2}{1000}.
$$
Because $f_0(t) \le f(t)$ for all $t \ge 0$, we have
$\widehat{E}_{f_0} \le \widehat{E}_f$, and equality holds for
the rhombic dodecahedron code.

For each potential function, we use an auxiliary function
$$
H(u,v,t) = c + \sum_{k = 0}^5 \langle F_k, \widehat{T}_k^n(u,v,t) \rangle,
$$
where the matrix $F_0$ is $5 \times 5$ (i.e., all entries outside of
the upper left $5 \times 5$ block are zero), $F_1$ and $F_2$ are $4
\times 4$, $F_3$ and $F_4$ are $3 \times 3$, and $F_5$ is $2 \times 2$.
We then optimize the bound obtainable from
Theorem~\ref{theorem:dualproj}.  To enforce the constraint that
$$
H(u,v,t) \le \frac{f_0\big(u^2\big)+f_0\big(v^2\big)+f_0\big(t^2\big)}{3},
$$
we write
$$
\frac{f_0\big(u^2\big)+f_0\big(v^2\big)+f_0\big(t^2\big)}{3}
- H(u,v,t) = \langle z(u,v,t), M z(u,v,t) \rangle,
$$
where $M \succeq 0$ is a $120 \times 120$ matrix and $z(u,v,t)$
is the vector consisting of all monomials in $u,v,t$ of degree
at most $7$.  Note that this condition is very strong (for
example, it forces the degree of the left side to be at most
$14$ and it forces the inequality to hold for all $(u,v,t) \in
\R^3$), but in fact it works.

The program SDPA-GMP can do arbitrary-precision semidefinite
programming, and CSDP can do sufficiently high-precision semidefinite
programming for our purposes (see \cite{Nplus} and \cite{Bor},
respectively). Using either of them, one can solve this semidefinite
program numerically and verify that the bound is nearly sharp. However,
getting a rigorous proof takes more work. We would like to round an
approximate solution to get an exact solution, but the rounding process
can violate the constraints at locations where there is equality, so it
must be done carefully.

To get a sharp bound, we will impose two types of conditions.
First, there are necessary constraints on the matrices $F_k$
from complementary slackness (i.e., the conditions under which
equality can hold in the proof of
Theorem~\ref{theorem:dualproj}). Specifically, for $0 \le k \le
5$, the inner product of $F_k$ with
\begin{align*}
&35 \delta_{k,0} J + 6 \widehat{T}^3_k\big(0,0,0\big)
+ 24 \widehat{T}^3_k\left(-\frac{1}{3},-\frac{1}{3},-\frac{1}{3}\right)\\
& \phantom{} +
36 \widehat{T}^3_k\left(-\frac{1}{3},\frac{1}{\sqrt{3}},\frac{1}{\sqrt{3}}\right)
+ 72 \widehat{T}^3_k\left(\frac{1}{\sqrt{3}},\frac{1}{\sqrt{3}},0\right)
+ 72 \widehat{T}^3_k\left(\frac{1}{3},\frac{1}{\sqrt{3}},\frac{1}{\sqrt{3}}\right)
\end{align*}
must vanish.  The coefficients and arguments of $\widehat{T}_k^3$ come
from the triple distribution of the rhombic dodecahedron code. Second,
we require that $H(u,v,t)$ must equal
$\left(f_0\big(u^2\big)+f_0\big(v^2\big)+f_0\big(t^2\big)\right)/3$ at
the triples of inner products that occur, namely the five arguments to
$\widehat{T}^3_k$ in the formula above, and all their first partial
derivatives must agree as well.  These conditions are linear in the
variables $c$, $F_0,\dots,F_5$, and $M$.

To perform the rounding correctly, we write the problem in a basis for
the set of all $c$, $F_0,\dots,F_5$, and $M$ that satisfy the
constraints listed in the previous paragraph, together with the
defining equations such as
$$
\frac{f_0\big(u^2\big)+f_0\big(v^2\big)+f_0\big(t^2\big)}{3}
- H(u,v,t) = \langle z(u,v,t), M z(u,v,t) \rangle.
$$
We then solve the semidefinite program numerically for the coefficients
in this basis, and we round the coefficients to eight or nine decimal
places to get an exact solution.  It is not guaranteed that this
rounding process will work: the problem is the semidefinite
constraints, because unexpected zero eigenvalues may become negative
due to the rounding. However, we do not run into that difficulty,
because there turn out to be no zero eigenvalues except for the ones
forced by the constraints we have built into our basis. Note that the
exact auxiliary function is far from unique, and our rounding strategy
makes essential use of this freedom.

Using this method in each of the five remaining cases proves optimality
in Theorem~\ref{theorem:univopt}.  Carrying out the calculations
requires a computer, most notably because $M$ is a $120 \times 120$
matrix, but all the calculations are completely rigorous.  In
particular, we use exact rational arithmetic, with no floating-point
approximations. See Appendix~\ref{appendix:data} for more details.

To complete the proof, we must verify uniqueness.  The Hermite
interpolation at more than two points (counting with multiplicity) of
an absolutely monotonic polynomial cannot be linear unless the original
polynomial is itself linear: otherwise, by Rolle's theorem, the second
derivative of their difference would have a root between the points,
and hence the second derivative of the original function would vanish,
which would violate absolute monotonicity. Thus, to prove uniqueness
for nonlinear potential functions, we just need to prove it for the
five cases $f(t) = t^2$, $t^3$, $t^3(t-1/9)$, $t^3(t-1/9)^2$, and
$t^3(t-1/9)^2(t-1/3)$. Let $f(t)$ be one of these, and let $c$,
$F_0,\dots,F_5$ and $M$ be as above (for the modified function $f_0$).

Because $f \ge f_0$ pointwise,
$$
H(u,v,t) \le \frac{f\big(u^2\big)+f\big(v^2\big)+f\big(t^2\big)}{3},
$$
with equality only when $u^2,v^2,t^2 \in \{0,1/9,1/3\}$. Thus,
there are only finitely many points at which equality is
possible, and we can simply check all the cases.  For $(u,v,t)
\in D$, the triple $(u,v,t)$ can occur in the global minimum
for energy only if equality holds (otherwise, something is lost
in the inequalities in the proof of
Theorem~\ref{theorem:dualproj}). It follows that the only
triples in $D$ that can occur in an optimal configuration are
$(0,0,0)$, $(-1/3,-1/3,-1/3)$, $(-1/3,1/\sqrt{3},1/\sqrt{3})$,
$(1/\sqrt{3},1/\sqrt{3},0)$, and $(1/3,1/\sqrt{3},1/\sqrt{3})$,
up to equivalence defined by permutations of the coordinates
and pairs of sign changes.  Let $N_1,\dots,N_5$ denote the
numbers of times each equivalence class occurs between a triple
of distinct points, with the numbering from $1$ to $5$ in the
order given here.

The numbers $N_1,\dots,N_5$ must satisfy several linear constraints.
Their sum must be $210$, and for $0 \le k \le 5$, the inner product of
$F_k$ with
\begin{align*}
&35 \delta_{k,0} J + N_1 \widehat{T}^3_k\big(0,0,0\big)
+ N_2 \widehat{T}^3_k\left(-\frac{1}{3},-\frac{1}{3},-\frac{1}{3}\right)\\
& \phantom{} +
N_3 \widehat{T}^3_k\left(-\frac{1}{3},\frac{1}{\sqrt{3}},\frac{1}{\sqrt{3}}\right)
+ N_4 \widehat{T}^3_k\left(\frac{1}{\sqrt{3}},\frac{1}{\sqrt{3}},0\right)
+ N_5 \widehat{T}^3_k\left(\frac{1}{3},\frac{1}{\sqrt{3}},\frac{1}{\sqrt{3}}\right)
\end{align*}
must vanish, for the same complementary slackness reason as above.  If
we solve these simultaneous equations, we find that $N_1=6$, $N_2=24$,
$N_3=36$, $N_4=72$, and $N_5=72$.

Because $N_1=6$, there is a unique triple of orthogonal points, up to
permutations.  Call them $e_1,e_2,e_3$ (as usual, we deal with
arbitrary lifts to the unit sphere).  The only possible inner products
between points in the code are $0$, $\pm 1/\sqrt{3}$, $\pm 1/3$, and
$\pm 1$, so every point other than $e_1,e_2,e_3$ must have inner
product $0$, $\pm 1/\sqrt{3}$, or $\pm 1/3$ with each $e_i$ (note that
$-1$ is ruled out by the uniqueness of $e_1,e_2,e_3$). Because
$e_1,e_2,e_3$ are orthogonal unit vectors, the sum of the squares of
the inner products with them must be $1$.  That means they must all be
$\pm 1/\sqrt{3}$.  In other words, with respect to the orthonormal
basis $e_1,e_2,e_3$, the code consists of the basis vectors together
with four points in the cube $(\pm 1/\sqrt{3}, \pm 1/\sqrt{3}, \pm
1/\sqrt{3})$.  It follows that it is the rhombic dodecahedron code.

Finally, for uniqueness as a spherical code, we use $f(t) =
t^3(t-1/9)^2(t-1/3)$.  Every code whose minimal distance is as large as
the rhombic dodecahedron code's minimal distance has $\widehat{E}_f \le
0$, and hence it minimizes $\widehat{E}_f$ (because the minimal energy
is $0$). Thus, uniqueness for this potential function implies
uniqueness as a spherical code.  This completes the proof of
Theorem~\ref{theorem:univopt}.

\section{Open questions} \label{section:open}

\subsection{Three-point bounds and generalizations}

One pressing question is why three-point bounds are not sharp more
often.  Of course, it is unreasonable to expect that sharp bounds will
ever be common in packing or energy minimization problems.  The
phenomena seem to be intrinsically complicated (see, for example,
\cite{BBCGKS}), and the most one can hope for is to prove optimality in
exceptional cases. However, even just for spheres, the two-point bounds
are sharp for three infinite families and a dozen sporadic cases (see
Table~1 in \cite{CK}), and they are sharp even more often in projective
spaces. Thus, it is surprising that only three sharp cases are known
for three-point bounds, excluding the cases where two-point bounds
already suffice. Surely there must be more examples, but so far we have
not found them.

It is natural to ask what sort of bounds are required to prove
optimality for an $N$-point configuration.  Of course $N$-point bounds
suffice, but only because in the process of optimizing over $N$-point
distributions they implicitly optimize over all $N$-point
configurations.  For what sorts of families of configurations might
$o(N)$-point bounds suffice?

It would  be very interesting to do explicit calculations with
four-point bounds.  This project would involve exceptionally
time-consuming calculations, but it is possible in principle and
perhaps in practice.  Gijswijt, Mittelmann, and Schrijver \cite{GMS}
have carried out the analogue for binary error-correcting codes, which
suggests that the continuous version may also be tractable.

If we had further examples of sharp three-point bounds, they might
suggest organizing principles that could lead to a deeper
understanding.  For example, for two-point bounds in two-point
homogeneous spaces, Levenshtein \cite{L2} proved a beautiful criterion
for being an optimal code in terms of strength as a design: any
$m$-distance set that is a $(2m-1)$-design (or even an antipodal
$(2m-2)$-design) is an optimal code.  In these cases, the two-point
bounds are sharp and can be understood conceptually, with no need for
numerical calculations. The same criterion applies more generally to
prove universal optimality (Theorems~1.2 and~8.2 in \cite{CK}).  Every
known case in which two-point bounds are sharp fits into this
framework, with one exception, namely the regular $600$-cell, which is
analyzed in \cite{Andreev,CK}.

We cannot even propose a conjectural generalization of Levenshtein's
criterion to three-point bounds.  However, we hope that some general
principle will explain the sharp cases, offer guidance for how to
locate more of them, and lead to proofs that involve no explicit
numerical computations.

Another question is whether there are good applications of three-point
bounds to potential functions that depend on triples of points, rather
than pairs.  Of course, there is no obstacle to writing down such
bounds, but it is not clear that there are any important examples. Pair
potentials are far more common in both physics and mathematics, and we
do not know which higher-order generalizations may be worthy of
investigation.

\subsection{Projective spaces}

Theorem~\ref{theorem:classify} settles the question of universal
optimality in $\R\Ps^2$, but as we will see in
Subsection~\ref{subsec:spheres}, there may be cases in which
three-point bounds are sharp but universal optimality does not hold. We
are unaware of any cases in projective space besides the rhombic
dodecahedron code in which three-point bounds are sharp and two-point
bounds are not, but it is difficult to imagine that it is the only
example.

It is especially intriguing that the three-point bounds prove universal
optimality for this code, and it would be fascinating to find other
such cases. It is unlikely that there are any in $\R\Ps^3$.  In
addition to the universal optima that occur in every dimension (up to
$n$ orthogonal lines in $\R^n$, or $n+1$ connecting the vertices of a
regular simplex to its centroid), there are at least two other
universally optimal line configurations in $\R^4$: $12$ lines through
opposite vertices of a regular $24$-cell and $60$ from a regular
$600$-cell.  In each of these cases, universal optimality follows from
two-point bounds: the last one can be proved using the techniques from
Section~7 of \cite{CK}, and all the others follow from Theorem~8.2 in
\cite{CK}. We know of no other candidates for universal optimality in
$\R\Ps^3$.

Line configurations in $\R^5$ are more promising.  In addition to the
generic cases of up to six lines, two-point bounds prove universal
optimality for a ten-line configuration from \cite{CHS} (by Theorem~8.2
in \cite{CK}).  In the subspace of $\R^6$ consisting of all points with
coordinate sum zero, this configuration consists of the lines through
all the permutations of $(1,1,1,-1,-1,-1)$.  If we include also the
lines through the permutations of $(5,-1,-1,-1,-1,-1)$, we get a
$16$-line configuration also studied in \cite{CHS}.  The orthoplex
bound is sharp for it, so it is an optimal projective code, but nothing
more is known about energy minimization or uniqueness.

\begin{conjecture} \label{conjecture:16in5}
The $16$-line configuration in $\R^5$ described above is universally
optimal and is the unique optimal $16$-point code in $\R\Ps^4$.
\end{conjecture}

The parallels between this code and the rhombic dodecahedron code are
noteworthy.  The seven lines in $\R^3$ can be constructed completely
analogously, using the permutations of the points $(1,1,-1,-1)$ and
$(3,-1,-1,-1)$.  Equivalently, both codes can be constructed by
starting with the lines through the vertices of a regular simplex
centered at the origin, and then filling in all the holes (i.e., the
lines at maximal distance from the original lines).

Both codes can be proved optimal using the orthoplex bound.
Nevertheless, the three-point bounds for potential energy do not seem
to settle Conjecture~\ref{conjecture:16in5}, although they suffice to
prove Theorem~\ref{theorem:univopt}.  Perhaps four-point bounds can
prove Conjecture~\ref{conjecture:16in5}, but the calculations required
to investigate this could be formidable.  The pattern certainly does
not continue to $\R^7$: the analogous configuration has
$\binom{8}{4}/2+8=43$ lines, but comparison with Table~4 from
\cite{CHS} shows that it is not even an optimal projective code.

\subsection{Spheres} \label{subsec:spheres}

Although in this paper we focus on real projective space, we have also
applied three-point bounds to potential energy minimization on spheres.
We have not found any sharp cases beyond the two identified by Bachoc
and Vallentin \cite{BV3}, namely the Petersen code in $\R^4$ (the ten
edge midpoints of a regular simplex) and the square antiprism in
$\R^3$. Using three-point bounds, they proved rigorously that the
Petersen code is an optimal code, and based on their calculations they
conjectured that the bounds are also sharp for the antiprism.

For ten points in $S^3$, we observe a remarkable phenomenon in the
three-point bounds for energy.  The Petersen code is not universally
optimal, because it sometimes has greater potential energy than the
code consisting of two regular pentagons on orthogonal planes in
$\R^4$.  However, the three-point bounds appear to be sharp in every
case:

\begin{conjecture} \label{conjecture:petersen}
For every completely monotonic potential function $f$ on $(0,4]$,
either the Petersen code or the orthogonal pentagon code minimizes the
energy $E_f$ among all ten-point codes in $S^3$.  Furthermore, the
three-point bounds are always sharp (for whichever code is optimal).
\end{conjecture}

In other words, the three-point bounds remain sharp throughout the
phase transition between the two ground states.  If true, this
represents an unprecedented phenomenon in coding and energy
minimization.

We believe that Conjecture~\ref{conjecture:petersen} should be
provable, although we have not been able to complete a proof.  The cone
of completely monotonic functions on $(0,4]$ is spanned by the
functions $f(r) = (4-r)^k$ for $k=0,1,2,\dots$ (see Theorem~9b in
\cite[p.~154]{W}). For $k \le 2$, both codes have the same energy and
the two-point bounds are sharp.  For $3 \le k \le 6$, the three-point
bounds are sharp for the orthogonal pentagons, and for $k \ge 7$ we can
prove that they are sharp for the Petersen code (by using the spherical
analogue of Corollary~\ref{corollary:reduce} to reduce to a finite
basis). To prove Conjecture~\ref{conjecture:petersen}, it would suffice
to prove it just for the functions $f(r) = (4-r)^j + \alpha_{j,k}
(4-r)^k$ with $3 \le j \le 6$ and $k \ge 7$, where $\alpha_{j,k}>0$ is
chosen to make the two codes have equal energy.  Unfortunately, we have
found it much more difficult to deal with these cases.

The same phenomenon as in Conjecture~\ref{conjecture:petersen} seems to
occur also for antiprisms. In that case, there is a one-parameter
family of configurations.  Each consists of two squares in parallel
planes and offset by a $45^\circ$ angle, but the distance between the
planes can vary.

\begin{conjecture} \label{conjecture:antiprism}
For every completely monotonic potential function $f$ on $(0,4]$, some
square antiprism minimizes the energy $E_f$ among all eight-point codes
in $S^2$. Furthermore, the three-point bounds are always sharp (for
whichever code is optimal).
\end{conjecture}

Because this conjecture involves a continuous family of optima, it may
be more difficult to prove than Conjecture~\ref{conjecture:petersen}.
It seems to require a new idea beyond what suffices for
Theorem~\ref{theorem:univopt} and our partial progress on
Conjecture~\ref{conjecture:petersen}.  The dependence of the optimal
energy on the potential function is also extraordinarily complicated,
because of the need to optimize over all square antiprisms. For
example, for the Coulomb potential, the minimal energy within the
family of antiprisms is a root of an irreducible polynomial of degree
$48$.


These conjectures suggest a broad generalization of the mechanism
behind the known proofs of universal optimality. Given a space $X$ and
a number of points $N$, call a potential function \emph{$k$-point
sharp} if there exists an $N$-point configuration in $X$ whose energy
under that potential function matches the $k$-point bounds. Subject to
one technical restriction, we conjecture that all completely monotonic
functions are $k$-point sharp if any one of them is:

\begin{conjecture} \label{conjecture:univ}
Suppose $X$ is a sphere or projective space and $N \in \N$. If there
exists a completely monotonic potential function that is $k$-point
sharp for $N$ points in $X$ and is not a polynomial, then every
completely monotonic potential function is $k$-point sharp for $N$
points in $X$. The same conclusion also holds if the $k$-point bounds
prove that an $N$-point configuration is an optimal code in $X$.
\end{conjecture}

For low-degree polynomials, there are numerous counterexamples: for
example, for each spherical $k$-design, two-point bounds prove that it
is optimal for $f(r) = (4-r)^k$.  However, that is the only loophole we
have found. Conjecture~\ref{conjecture:univ} is very strong, and
perhaps we have been misled by the few known examples with $k > 2$. It
may hold in these examples merely because the optimal structures are in
some sense of comparable complexity.  However,
Conjecture~\ref{conjecture:univ} seems to be the simplest explanation
of the available evidence.

The most dramatic test case will be $24$ points in $S^3$.  The vertices
of a regular $24$-cell are almost certainly the unique optimal code,
but they are not universally optimal \cite{CCEK}.  Instead, there is a
one-parameter family of competing configurations that improve on the
$24$-cell for some potential functions.  Numerical evidence suggests
that either the $24$-cell or one of these competitors is always
optimal. If that is true, and Conjecture~\ref{conjecture:univ} is true
as well, then $k$-point bounds cannot settle the optimality of the
$24$-cell without also dealing with its more exotic competitors.  If
they can accomplish both with a small value of $k$, it will be truly
remarkable.

\appendix
\section{The orthoplex bound}
\label{appendix:orthoplex}

In this appendix, we explain the orthoplex bound from
\cite{CHS} in somewhat different terms than those used there
(although it is essentially the same approach), and we discuss
some examples and generalizations.

The orthoplex bound applies to codes in any Grassmannian, but for real
projective space it is particularly simple.  It is based on the map
from $\R^n$ to its symmetric square $\Sym^2 \R^n$ that takes $x$ to $x
\otimes x$.  Under the natural inner product on $\Sym^2 \R^n$, $\langle
x \otimes x, y \otimes y\rangle = \langle x,y \rangle^2$.  Let
$e_1,\dots,e_n$ be an orthonormal basis of $\R^n$, and let $E =
\sum_{i=1}^n e_i \otimes e_i$.  Then for all $x \in \R^n$, $|x|^2 =
\langle x \otimes x, E \rangle$.  Thus, the unit sphere $S^{n-1}$ is
mapped into the cross section of the unit sphere in $\Sym^2 \R^n$ that
has inner product $1$ with the vector $E$ (note that $|E|^2 = n$).

For $x \in S^{n-1}$, define $\varphi(x) = (x \otimes x -
E/n)/\sqrt{1-1/n}$.  (In other words, start with $x \otimes x$, project
orthogonally to $E$, and rescale to get a unit vector.) Then
$|\varphi(x)|^2 = 1$ and $\langle \varphi(x), \varphi(y) \rangle =
(\langle x,y \rangle^2-1/n)/(1-1/n)$, which implies that
$\varphi(x)=\varphi(y)$ if and only if $x = \pm y$.  The image of
$\varphi$ is orthogonal to $E$ and is thus contained in the unit sphere
of a Euclidean space of $n(n+1)/2-1$ dimensions.

Under $\varphi$, an $(n,N,t)$ antipodal code is mapped to an
$$\big(n(n+1)/2-1, N/2, (t^2-1/n)/(1-1/n)\big)$$ spherical code.  Thus, any
bound on such spherical codes yields a bound on the original antipodal
code. Linear programming bounds are useless in this context, because
any positive-definite function on the higher-dimensional sphere pulls
back under $\varphi$ to a positive-definite function on $S^{n-1}$ (so
there is no need to apply $\varphi$ in the first place). However, other
bounds may prove fruitful. For example, for $m+1 < k \le 2m$, no
$k$-point code in $S^{m-1}$ can achieve an angle of better than $\pi/2$
(see \cite{R} or Theorem~6.2.1 in \cite{BJr}). It follows that when
$n(n+1)/2 < N/2 \le n(n+1)-2$, no $N$-point antipodal code in $S^{n-1}$
can have maximal inner product less than $1/\sqrt{n}$, since otherwise
its image under $\varphi$ would have maximal inner product less than
$0$.  This is the orthoplex bound for projective codes.

It is not obvious that the orthoplex bound is ever sharp, but in fact
it is sharp in several important cases, for example when
$$
(N/2,n) \in
\{(7,3),(11,4),(12,4),(16,5),(22,6),(46,9),(47,9),(48,9)\}.
$$
These cases are all taken from \cite{CHS}.
The $(7,3)$ case is of
course   the rhombic dodecahedron code.
Another noteworthy case is $N/2 = n(n+2)/2$ with $n$ a power of $4$
(which generalizes the $(12,4)$ case), due to Levenshtein \cite{L1}. In
this case, the orthoplex bound shows that optimality still holds when
up to $n/2 - 1$ lines are omitted.

The orthoplex bound does not imply uniqueness. One problem is that when
$N/2 < n(n+1)-2$, there are continuous families of spherical codes with
maximal inner product $0$.  To prove uniqueness, one would have to
determine which such codes are in the image of $\varphi$, which seems
to be a subtle question.  Even when $N/2 = n(n+1)-2$, the image of
$\varphi$ is uniquely determined up to isometries of the sphere
containing it, but it does not obviously follow that the original code
is also unique.

Uniqueness is not merely difficult to prove, but in fact sometimes
false.  For example, the unique $(4,24,1/2)$ antipodal code is the
regular $24$-cell (uniqueness follows easily from the two-point
bounds). As one can see in the list of examples above, the orthoplex
bound proves not only that it is optimal, but also that removing one
pair of antipodal points still yields an optimal antipodal code of size
$22$. However, that is not the only $(4,22,1/2)$ antipodal code.
Another one is given by the vertices $(\pm 1/\sqrt{3}, \pm 1/\sqrt{3},
\pm 1/\sqrt{3},0)$ of a cube, a pair $(0,0,0,\pm 1)$ of points
antipodal to all of them, and the two octahedra
$$
(\pm \sqrt{3}/2, 0, 0, \pm 1/2), \quad
(0, \pm \sqrt{3}/2, 0, \pm 1/2), \quad (0, 0, \pm \sqrt{3}/2, \pm 1/2)
$$
in the $3$-dimensional affine subspaces with fourth coordinate $1/2$ or
$-1/2$. This example, which appears to be new, is a $(4,22,1/2)$
antipodal code that is not contained in the $24$-cell, because inner
products of $\pm 1/3$ and $\pm 1/4$ occur.  We do not know of any other
$(4,22,1/2)$ antipodal codes, but we have not searched thoroughly
enough to be confident that no others exist.

Note that the second Gegenbauer polynomial $P^n_2$ for $S^{n-1}$ is
given by $P^n_2(t) = (t^2-1/n)/(1-1/n)$.  Thus, the map $\varphi$
satisfies $\langle \varphi(x), \varphi(y) \rangle = P^n_2(\langle x,y
\rangle)$. This is not a coincidence, and in fact the orthoplex bound
can be generalized to substitute any positive-definite function $f$ for
$P^n_2$ if it is normalized to satisfy $f(1)=1$.  For individual
Gegenbauer polynomials this follows immediately from the reproducing
kernel property (see Subsection~2.2 of \cite{CK} for a review of the
theory), and nonnegative linear combinations can be obtained by taking
direct sums of the corresponding spaces. Unfortunately, we have been
unable to find any cases other than $f=P^n_2$ that lead to sharp bounds
on antipodal or general spherical codes.

\section{Universal optima in $\mathbb{R}\mathbb{P}^2$}
\label{appendix:rp2}

The proof of completeness for the list in
Theorem~\ref{theorem:classify} is based on a theorem of Leech
\cite{Lee}.  He classified balanced configurations in $S^2$, i.e.,
point configurations that are in equilibrium for every pairwise force
law depending only on distance.  Equivalently, in the space of
configurations they are critical points \pagebreak for energy for every potential
function.  This is a necessary condition for universal optimality in
$S^2$, and for universal optimality in $\R\Ps^2$ it is necessary that
the corresponding antipodal configuration in $S^2$ be balanced.  (If it
is not balanced, then the net forces inherit the antipodal symmetry and
descend to projective space.)

One obvious balanced configuration is a ring of equally spaced points
around the equator, and it remains balanced if one includes the north
and south poles as well.  In addition, each regular polyhedron gives
several balanced configurations.  Its vertices, face centers, and edge
midpoints each form a balanced configuration, and every combination of
them (vertices and faces, vertices and edges, faces and edges, or all
three) is also balanced when they are rescaled to lie on the same
sphere. Leech proved that this list of balanced configurations in $S^2$
is complete. By contrast, the high-dimensional case is not so simple,
and there exist balanced configurations with no symmetry at all
\cite{CEKS}.

To classify the universal optima in $\R\Ps^2$, one need only examine
the antipodal balanced configurations in $S^2$.  The cases with a ring
around the equator are easily dealt with: whenever the ring has more
than four points in it and the potential function is sufficiently
steep, one can lower the energy by moving one point slightly off the
equator (and of course moving the antipodal point correspondingly). The
only difficulty is dealing with the finitely many cases that come from
regular polyhedra.

\begin{table}
\caption{Balanced configurations that must be eliminated.}
\label{table:polyhedra}
\begin{center}
\begin{tabular}{ccccc}
Polyhedron& Combination & $N$ & Cosine of minimal angle & Record\\
\hline
Octahedron & E & $6$ & $1/2$ & $0.447\dots$\\
Octahedron & VE & $9$ & $1/\sqrt{2} = 0.707\dots$ & $0.669\dots$\\
Octahedron & EF & $10$ & $\sqrt{2/3} = 0.816\dots$ & $0.686\dots$\\
Octahedron & VEF & $13$ & $\sqrt{2/3} = 0.816\dots$ & $0.768\dots$\\
Icosahedron & E & $15$ & $(1+\sqrt{5})/4 = 0.809\dots$ & $0.786\dots$\\
Icosahedron & F & $10$ & $\sqrt{5}/3 = 0.745\dots$ & $0.686\dots$\\
Icosahedron & VE & $21$ & $\sqrt{(5+\sqrt{5})/10} = 0.850\dots$ & $0.846\dots$\\
Icosahedron & VF & $16$ & $\sqrt{(5+2\sqrt{5})/15} = 0.794\dots$ & same\\
Icosahedron & EF & $25$ & $\sqrt{(3+\sqrt{5})/6} = 0.934\dots$ & $0.872\dots$\\
Icosahedron & VEF & $31$ & $\sqrt{(3+\sqrt{5})/6} = 0.934\dots$ & $0.894\dots$\\
\\ 
\end{tabular}
\end{center}
\end{table}

Table~\ref{table:polyhedra} shows a list of the cases that must be
eliminated.  It omits the actual universal optima from
Theorem~\ref{theorem:classify} as well as any duplicate cases (for
example, from duality).  The first four columns specify the polyhedron,
which combination of vertices, edges, and faces is being used, the
number $N$ of lines, and the cosine of the minimal angle between them.
The final column tells the smallest such cosine known for a
configuration of this size, from Table~1 in \cite{CHS}.  In every case
except $N=16$, the record from \cite{CHS} is better than the balanced
configuration, which is therefore not universally optimal since it is
not even an optimal projective code.  However, the balanced
configuration with $N=16$ is probably an optimal code.

Nevertheless, the $N=16$ case is not universally optimal.  It is a
$2$-design in $\R\Ps^2$ but not a $3$-design.  However, a $16$-point
$3$-design exists: it is equivalent to an antipodal $32$-point
spherical $7$-design in $S^2$, and such a design is constructed in
Section~4 of \cite{HS}.  It follows that the $16$-line configuration is
not universally optimal, and hence the list in
Theorem~\ref{theorem:classify} is complete.

\section{Universal optimality of the regular $600$-cell}
\label{appendix:600cell}

\begin{theorem}[Cohn and Kumar \cite{CK}]
The $120$ vertices of the regular $600$-cell are universally optimal in
$S^3$.
\end{theorem}

\begin{proof}[Sketch of proof]
Section~7 of \cite{CK} describes a construction that applies two-point
bounds for energy.  It gives a sharp bound for every absolutely
monotonic potential function, and in any given case it is easy to check
that it works, but the proof that it works for all of them is fairly
elaborate.  We will use Corollary~\ref{corollary:reduce} (or,
technically, its analogue for spheres) to reduce the problem to a
finite calculation.

There are eight inner products between distinct points in the
$600$-cell, namely $-1$, $0$, $\pm 1/2$, and $(\pm 1 \pm \sqrt{5})/4$.
Call them $t_1,\dots,t_8$ with $ -1 = t_1 < t_2 < \dots < t_8. $ Let
$T$ be the multiset $\{t_1,t_2,t_2,t_3,t_3,\dots,t_8,t_8\}$ (note that
$\mult_T(t_1)=1$).  Now Corollary~\ref{corollary:reduce} reduces the
problem to a finite number of cases. In each case but the last, it is
not hard to check that the approach in Section~7 of \cite{CK} proves a
sharp bound, but it fails in the last case, namely
$$
f(t) = (t-t_1)(t-t_2)^2\dots (t-t_7)^2 (t-t_8).
$$
The failure does not come as a surprise, because of course this
function is not absolutely monotonic.  Nevertheless, the approach can
be salvaged in this case.  In the language used at the top of page~127
in \cite{CK}, one should choose an auxiliary polynomial of degree $18$
that agrees with $f$ to order $2$ at $-1$ (in addition to the other
conditions listed in \cite{CK}). It is then straightforward to check
that this polynomial proves a sharp bound for $\widetilde{E}_f$.
\end{proof}

This proof is still not quite conceptual, because it relies on a fair
amount of computation, but it requires less computation and is
substantially simpler than the proof in \cite{CK}.  The same approach
of using the spherical analogue of Corollary~\ref{corollary:reduce}
works for all the universal optima studied in \cite{CK}; in fact, it is
in a sense equivalent to the approach used there, but this method was
not made explicit in \cite{CK} and was not applied to the $600$-cell.

\section{Computer calculations}
\label{appendix:data}

Computer algebra files for checking the proof of
Theorem~\ref{theorem:univopt} can be downloaded as part of the source
files for this paper in the \texttt{arXiv.org} e-print archive, where
it is available as \texttt{arXiv:1103.0485}. In addition to the \LaTeX\
source for the paper, downloading the source files produces eight plain
text files: \texttt{definitions.txt}, \texttt{optimal.txt},
\texttt{unique.txt}, and \texttt{data1.txt} through \texttt{data5.txt}.
These files contain code for the PARI/GP computer algebra system
\cite{PARI}, which is freely available.  The code is not especially
elaborate and should be straightforward to adapt to other systems.
The files are also available from
\url{http://dx.doi.org/10.1090/S0894-0347-2012-00737-1}.

Each of the data files contains the exact values of $c$,
$F_0,\dots,F_5$, and $M$ for one of the five potential functions
considered in the proof of Theorem~\ref{theorem:univopt}. Specifically,
it defines a constant $c$ and a list $L$ of six matrices, namely
$F_0,\dots,F_5$, and $M$.  Note that in PARI format, \texttt{[a,b;c,d]}
denotes the matrix
$$
\begin{pmatrix}
a & b\\
c & d
\end{pmatrix}.
$$
In the proof of the theorem, we did not specify the order of the
monomials $u^i v^j t^k$ in the vector $z(u,v,t)$, or equivalently the
order of the rows and columns of $M$; in the data files, we order them
lexicographically by their exponent vectors $(i,j,k)$.

The file \texttt{definitions.txt} defines the ultraspherical
polynomials and the matrices $\widehat{S}_k^n$ and $\widehat{T}_k^n$.
These definitions are read by the remaining two files.

The file \texttt{optimal.txt} completes the proof of universal
optimality. Specifically, it checks that all the matrices are symmetric
and positive semidefinite, it checks the sum of squares decomposition
$$
\frac{f_0\big(u^2\big)+f_0\big(v^2\big)+f_0\big(t^2\big)}{3} - H(u,v,t)
= \langle z(u,v,t), M z(u,v,t) \rangle,
$$
and it checks that the bound is sharp.  To show that a $d \times d$
matrix is positive semidefinite, it computes the characteristic
polynomial $p(x)$ and verifies that $(-1)^d p(-x)$ has no negative
coefficients. All the calculations are done using exact rational
arithmetic, so the resulting proof is rigorous.

Finally, the file \texttt{unique.txt} completes the proof of uniqueness
and thus of Theorem~\ref{theorem:univopt}.  It finds all $(u,v,t)$ such
that
$$
\frac{f\big(u^2\big)+f\big(v^2\big)+f\big(t^2\big)}{3} =
H(u,v,t),
$$
and it solves a system of simultaneous equations to determine how many
times such a triple must occur in an optimal configuration (i.e., to
compute the numbers $N_1,\dots,N_5$).  It is convenient to use the
number $1/\sqrt{3}$, but we must avoid floating point arithmetic to
keep round-off error from becoming a problem.  To do so, we use
polynomials in a variable $s$ and work modulo $s^2-1/3$. As above, all
the calculations are then carried out exactly.

\section*{Acknowledgements}

We thank Christine Bachoc, Abhinav Kumar, and Frank Vallentin for
helpful discussions and Nathan Kaplan and the anonymous referees for
their feedback on the manuscript.


\end{document}